\documentclass[english]{amsart}

\usepackage{amsmath}
\usepackage{amssymb}
\usepackage{amsthm}
\usepackage{amscd}
\usepackage{babel}
\usepackage[latin1]{inputenc}
\usepackage{graphicx}

\newtheorem{teor}{Theorem}

\newtheorem{prop}{Proposition}

\newtheorem{lem}{Lemma}

\theoremstyle{definition}

\renewcommand{\subjclassname}{AMS \textup{2010} Mathematics Subject
Classification\ }

\author{J.M. Grau}
\address{Departamento de Matemáticas, Universidad de Oviedo\\ Avda. Calvo Sotelo s/n, 33007 Oviedo, Spain}
\email{grau@uniovi.es}

\author{A. M. Oller-Marc\'{e}n}
\address{Centro Universitario de la Defensa\\ Ctra. Huesca s/n, 50090 Zaragoza, Spain} \email{oller@unizar.es}

\author{C. Tasis}
\address{Departamento de Matemáticas, Universidad de Oviedo\\ Avda. Calvo Sotelo s/n, 33007 Oviedo, Spain}
\email{ctasis@uniovi.es}

\title{On the diameter of the commuting graph of the full matrix ring over the real numbers}
\begin{document}

\maketitle
\subjclassname{05C50; 15A27}


\begin{abstract} 
In a recent paper C. Miguel proved that the diameter of the commuting graph of the matrix ring $\mathrm{M}_n(\mathbb{R})$ is equal to $4$ either if $n=3$ or $n>4$. But the case $n=4$ remained open, since the diameter could be $4$ or $5$. In this work we close the problem showing that also in this case the diameter is equal to $4$.
\end{abstract}

\section{Introduction}

For a ring $R$, the \textit{commuting graph} of $R$, denoted by $\Gamma(R)$, is a simple undirected graph whose vertices are all non-central elements of $R$, and two distinct vertices $a$ and $b$ are adjacent if and only if $ab = ba$. The commuting graph was introduced in \cite{Ak1} and has been extensively studied in recent years by several authors \cite{Ak2,Ak3,Ak4,Ak5,Araujo,doli,mo, omi}.

In a graph $G$, a path $\mathcal{P}$ is a sequence of distinct vertices $(v_1,\cdots v_k)$ such that every two consecutive vertices are adjacent. The number $k-1$ is called the length of $\mathcal{P}$. For two vertices $u$ and $v$ in a graph $G$, the distance between $u$ and $v$, denoted by $d(u,v)$, is the length of the shortest path between $u$ and $v$, if such a path exists. Otherwise, we define $d(u,v) =\infty$. The diameter of a graph $G$ is defined $$\textrm{diam}(G) = \sup	\{d(u,v) : \textrm{ \textit{u} and \textit{v} are distinct vertices of $G$\}}.$$
A graph $G$ is called connected if there exists a path between every two distinct vertices of $G$.

Much research has been conducted regarding the diameter of commuting graphs of certain classes of rings \cite{Ak3,doli,Dolan1,Giu}. In the case of matrix rings over an algebraically closed field $\mathbb{F}$, $\mathrm{M}_n(\mathbb{F})$, it was proved \cite{Ak3} that the commuting graph with $n>2$ is connected and its diameter is always equal to four; while if $n = 2$ the commuting graph is always disconnected \cite{Ak5}. On the other hand, if the field $\mathbb{F}$ is not algebraically closed, the commuting graph $\Gamma(\mathrm{M}_n(\mathbb{F}))$ may be disconnected for an arbitrarily large integer $n$ \cite{Ak4}. However, for any field $\mathbb{F}$ and $n\geq 3$, if $\Gamma(\mathrm{M}_n(\mathbb{F}))$ is connected, then the diameter is between four and six \cite{Ak3}. Moreover, this diameter is conjectured to be at most 5 and if $n=p$ is prime it is proved that the diameter is, in fact, 4. Quite recently, C. Miguel \cite{cel} has verified this conjecture in the case $\mathbb{F}=\mathbb{R}$ proving the following result.

\begin{teor}
Let $n\geq 3$ be any integer. Then, $\textrm{diam}(\Gamma(\mathrm{M}_n(\mathbb{R})))=4$ for $n\neq 4$ and $4\leq \textrm{diam}(\Gamma(\mathrm{M}_4(\mathbb{R})))\leq 5$.
\end{teor}

Unfortunately, this result left open the question wether $\textrm{diam}(\Gamma(\mathrm{M}_4(\mathbb{R})))$ is $4$ or $5$. In this paper we solve this open problem. Namely, we will prove the following result.

\begin{teor}
For every $n\geq 3$, $\textrm{diam}(\Gamma(\mathrm{M}_n(\mathbb{R})))=4$.
\end{teor}

\section{On the diameter of $\Gamma(\mathrm{M}_n(\mathbb{R})$}

Before we proceed, let us introduce some notation. If $a,b\in\mathbb{R}$, we define the matrix $A_{a,b}$ as
$$A_{a,b}:=\begin{pmatrix} a & b\\ -b & a\end{pmatrix}.$$
Now, given two matrices $X,Y\in M_2(\mathbb{R})$, we define
$$X\oplus Y:=\begin{pmatrix} X & 0\\ 0 & Y\end{pmatrix}\in \mathrm{M}_4(\mathbb{R}).$$
Finally, two matrices $A,B\in \mathrm{M}_2(\mathbb{R})$ are similar ($A\sim B$) if there exists a regular matrix $P$ such that $P^{-1}AP=B$.

As we have pointed out in the introduction, in \cite[Theorem 1.1.]{cel} it is proved that the diameter of $\Gamma(\mathrm{M}_n(\mathbb{R}))$ is equal to 4 if $n\geq 3, n\neq 4$ and that $4\leq\textrm{diam}(\Gamma(\mathrm{M}_4(\mathbb{R})))\leq 5$. The proof given in that paper relies on the possible forms of the Jordan canonical form of a real matrix. In particular, it is proved that the distance between two matrices $A,B\in \mathrm{M}_4(\mathbb{R})$ is at most 4 unless we are in the situation where $A$ and $B$ have no real eigenvalues and only one of them is diagonalizable over $\mathbb{C}$. In other words, the case when
$$A\sim \begin{pmatrix} A_{a,b} & 0 \\ 0 & A_{c,d} \end{pmatrix},\quad B\sim \begin{pmatrix} A_{s,t} & I_2 \\ 0 & A_{s,t} \end{pmatrix}.$$

The following result will provide us the main tool to prove that the distance between $A$ and $B$ is at most $4$ also in the previous setting. It is true for any division ring $D$.

\begin{prop}
\label{prop}
Let $A,B\in \mathrm{M}_n(D)$ matrices such that $A^2=A$ and $B^2=0$. Then, there exists a non-scalar matrix commuting with both $A$ and $B$.
\end{prop}
\begin{proof}
Since $A^2=A$; i.e., $A(I-A)=(I-A)A=0$, then one of nullity $A$ or nullity $(I-A)$ is at least $n/2$. Since $I-A$ is also idempotent and a matrix commutes with $A$ if and only if it commutes with $I-A$ we can assume that nullity $A\geq n/2$. On the other hand, since $B^2=0$, it follows that nullity $B\geq n/2$.

Now, if $\textrm{Ker} L_A\cap\textrm{Ker} L_B\neq\{0\}$ and $\textrm{Ker} R_A\cap\textrm{Ker} R_B\neq\{0\}$ we can apply \cite[Lemma 4]{Ak3} and the result follows. Hence we assume that $\textrm{Ker} L_A\cap\textrm{Ker} L_B=\{0\}$ (if it was $\textrm{Ker} R_A\cap\textrm{Ker} R_B=\{0\}$ we could consider $A^t$ and $B^t$). Note that, in these conditions, $n=2r$ and nullity $A$ and nullity $B$ are equal to $r$.

Let $\mathcal{B}_1$ and $\mathcal{B}_2$ be bases for $\textrm{Ker} L_A$ and $\textrm{Ker} L_B$, respectively, and consider $\mathcal{B}=\mathcal{B}_1\cup\mathcal{B}_2$ a basis for $D^n$. Since $A$ is idempotent, it follows that $D^n=\textrm{Ker} L_A\oplus \textrm{Im} L_A$.

We want to construct the matrix of $L_A$ in the basis $\mathcal{B}$. To do so, if $v\in\mathcal{B}_2$, we write $v=a+a'$ with $a\in\textrm{Ker}L_A$ and $a'\in\textrm{Im}L_A$. Hence, $Av=Aa+Aa'=0+A(Aa'')=Aa''=a'=-a+v$. Since it is clear that $Av=0$ for every $v\in\mathcal{B}_1$ we get that the matrix of $L_A$ in the basis $\mathcal{B}$ is of the form
$$\begin{pmatrix} 0 & A'\\ 0 & I_r\end{pmatrix},$$
with $A'\in \mathrm{M}_r(D)$.

Now we want to construct the matrix of $L_B$ in the basis $\mathcal{B}$. Clearly $Bv=0$ for every $v\in\mathcal{B}_2$. Now, let $w\in\mathcal{B}_1$. Then, $Bw=w_1+w_2$ with $w_1\in\textrm{Ker} L_A$ and $w_2\in\textrm{Ker}L_B$. Hence, $0=B^2w=Bw_1$ and $w_1\in\textrm{Ker} L_A\cap\textrm{Ker} L_B=\{0\}$. Thus, the matrix of $L_B$ in the basis $\mathcal{B}$ is of the form:
$$\begin{pmatrix} 0 & 0\\ B' & 0\end{pmatrix},$$
with $A'\in \mathrm{M}_r(D)$.

As a consequence of the previous work we can find a regular matrix $P$ such that:
$$PAP^{-1}=\begin{pmatrix} 0 & A'\\ 0 & I_r\end{pmatrix},\quad PBP^{-1}=\begin{pmatrix} 0 & 0\\ B' & 0\end{pmatrix}.$$

Now, if $A'B'\neq B'A'$ we can consider the matrix
$$P^{-1}(A'B'\oplus B'A')P=P^{-1}\begin{pmatrix} A'B' & 0\\ 0 & B'A'\end{pmatrix}P,$$
which is clearly non-scalar and commutes with $A$ and $B$. On the other hand, if $A'$ and $B'$ commute, we can find a non-scalar matrix $S\in \mathrm{M}_r(D)$ commuting with both $A'$ and $B'$. Therefore $P^{-1}(S\oplus S)P$ commutes with both $A$ and $B$ and the proof is complete.
\end{proof}

In addition to this result, we will also need the following technical lemmata.

\begin{lem}
If $A\sim \begin{pmatrix} A_{a,b} & 0 \\ 0 & A_{c,d} \end{pmatrix}$, then there exists an idempotent non-scalar matrix $M$ such that $AM=MA$.
\end{lem}
\label{l1}
\begin{proof}
$A=P^{-1}  \begin{pmatrix} A_{a,b} & 0 \\ 0 & A_{c,d} \end{pmatrix}P$ for some regular $P\in \mathrm{M}_4(\mathbb{R})$. Hence, it is enough to consider $M=P^{-1}  \begin{pmatrix} 0 & 0 \\ 0 & I_2 \end{pmatrix}P$.
\end{proof}

\begin{lem}
\label{l2}
If $B\sim \begin{pmatrix} A_{s,t} & I_2 \\ 0 & A_{s,t} \end{pmatrix}$, then there exists a non-scalar matrix $N$ such that $N^2=0$ and $BN=NB$.
\end{lem}
\begin{proof}
$B=P^{-1}\begin{pmatrix} A_{s,t} & I_2 \\ 0 & A_{s,t} \end{pmatrix}P$ for some regular $P\in \mathrm{M}_4(\mathbb{R})$. Hence, it is enough to consider $N=P^{-1}\begin{pmatrix} 0 & I_2 \\ 0 & 0 \end{pmatrix} P$.
\end{proof}

We are now in the condition to prove the main result of the paper.

\begin{teor}
The diameter of $\Gamma(\mathrm{M}_4(\mathbb{R}))$ is four.
\end{teor}
\begin{proof}
In \cite{cel} it was proved that $d(A,B)\leq 4$ for every $A,B\in \mathrm{M}_4(\mathbb{R})$ unless $A\sim \begin{pmatrix} A_{a,b} & 0 \\ 0 & A_{c,d} \end{pmatrix}$ and $B\sim \begin{pmatrix} A_{s,t} & I_2 \\ 0 & A_{s,t} \end{pmatrix}$. Hence, we only focus on this case.

By Lemma \ref{l1} there exists an idempotent non-scalar matrix $M$, such that $AM=MA$. Also, by Lemma \ref{l2}, there exists a non-scalar matrix $N$ such that $N^2=0$ and $NB=BN$. Finally, Proposition \ref{prop} implies that there exists a non-scalar matrix $X$ that commutes both with $M$ and $N$.

Thus, we have found a path $(A,M,X,N,B)$ of length $4$ connecting $A$ and $B$ and the result follows.
\end{proof}

\end{document}